\documentclass[12pt]{amsart}

\usepackage[utf8]{inputenc}

\usepackage{amsfonts, amsmath, amssymb, amsthm, bbm, color, enumerate, graphicx, mathtools, tikz, hyperref, relsize,enumitem}
\usepackage[numbers,sort]{natbib}
\usepackage{comment}
\usepackage[margin=30mm]{geometry}

\allowdisplaybreaks

\newtheorem{theorem}{Theorem}[section]
\newtheorem{corollary}[theorem]{Corollary}
\newtheorem{lemma}[theorem]{Lemma}
\newtheorem{proposition}[theorem]{Proposition}

\theoremstyle{definition}
\newtheorem{definition}{Definition}
\newtheorem*{example}{Example}
\newtheorem*{counterexample}{Counterexample}

\newtheorem{assumption}{Assumption}

\newtheorem*{remark}{Remark}
\newtheorem{step}{Step}

\let\plainqed\qedsymbol
\newcommand{\claimqed}{$\lrcorner$}

\hypersetup{
	colorlinks=true,
	linkcolor=blue,          
	citecolor=red,       
	filecolor=red,   
	urlcolor=red, 
	pdftitle={},
	pdfauthor={},
	pdfsubject={},
	pdfkeywords={},
	linktocpage=true
}

\usepackage{fancyhdr}

\usepackage{natbib}

\usepackage{amssymb}
\usepackage{amsmath}
\usepackage{amsthm}

\newcommand{\eps}{\varepsilon}

\newcommand{\NN}{{\mathbb N}}

\newcommand{\RR}{{\mathbb R}}

\newcommand{\EE}{{\mathbb E}}

\newcommand{\I}{{\mathbb I}}
\newcommand{\PP}{{\mathbb P}}

\newcommand{\calB}{{\mathcal B}}

\newcommand{\calC}{{\mathcal C}}

\newcommand{\calJ}{{\mathcal J}}
\newcommand{\calK}{{\mathcal K}}
\newcommand{\calL}{{\mathcal L}}

\newcommand{\calO}{{\mathcal O}}
\newcommand{\calP}{{\mathcal P}}

\newcommand{\calS}{{\mathcal S}}

\newcommand{\veps}{\varepsilon}
\newcommand{\be}{\begin{equation}}
	\newcommand{\ee}{\end{equation}}

\numberwithin{equation}{section}

\newcommand{\editcol}[1]{{\color{black}#1}}
\makeatletter
\def\namedlabel#1#2{\begingroup
	#2%
	\def\@currentlabel{#2}%
	\phantomsection\label{#1}\endgroup
}
\makeatother

\title[Feller and continuous viscosity solutions]{A selection procedure for extracting the unique Feller weak solution of degenerate diffusions.}
\author[Sumith Reddy]{Anugu Sumith Reddy}

\address{Anugu Sumith Reddy\\
Dept.\ of Mathematics \\
	IIT Bombay, Powai\\
	Mumbai 400076, India.}
\email{anugu.reddy@math.iitb.ac.in} 
\author[Vivek S. Borkar]{Vivek S. Borkar}
\address{ Vivek S. Borkar\\
Dept.\  of Electrical Engg.\\
IIT Bombay, Powai\\
Mumbai 400076, India. }
\email{borkar.vs@gmail.com}

\keywords{Degenerate diffusions, Markov selection, Feller solution,  small noise limit, non-uniqueness of weak solution, backward Kolmogorov equation, uniqueness of continuous viscosity solutions}
\subjclass[2010]{60H10, 60J25, 34F05, 35K65, 35D40, 49L25}

\begin{document}
	
	\begin{abstract}
In this work, we show that for the martingale problem for a class of degenerate diffusions with bounded continuous drift and diffusion coefficients, the small noise limit
of non-degenerate approximations leads to a unique Feller limit. The proof uses the theory of viscosity solutions applied to the associated backward Kolmogorov equations. Under appropriate conditions on drift and diffusion coefficients, we will establish a comparison principle and a one-one correspondence between Feller solutions to the martingale problem and continuous viscosity solutions of the associated Kolmogorov equation. This work can be considered as an extension to the work in \cite{borkar2010new}.
	\end{abstract}
	
	\maketitle
	\section{Introduction}
		
	Let $\calC\doteq \calC([0,T],\RR^p)$ denote the space of continuous $\RR^p$- valued functions endowed with a supremum norm and let $\calB\doteq \calB(\calC)$ denote the Borel $\sigma$-algebra of $\calC$. Consider a stochastic differential equation on $\RR^n$ given by
	\begin{align}\label{eq:sde}
		dX_t=b(X_t)dt +\sigma (X_t) dW_t,\quad X_0\sim \mu 
	\end{align}
	where, $X_t\in \RR^n$, $b(\cdot): \RR^n\rightarrow \RR^n$ and $\sigma(\cdot):\RR^n\rightarrow \RR^{n\times m}$ with $W_{(\cdot)}$ being a Brownian motion on $\RR^m$. We define $\PP_X$ to be the law of process $X$ and $\EE_X$ to be the corresponding expectation.

	If we assume that the coefficients of~\eqref{eq:sde} are bounded continuous, then it is well known that the notion of weak solution and the notion of solution to the corresponding martingale problem (due to Stroock and Varadhan) of~\eqref{eq:sde} are equivalent \cite[Proposition 4.11]{karatzas2012brownian} and that~\eqref{eq:sde} has a solution to the corresponding martingale problem \cite[Theorem 6.1.7]{stroock1979multidimensional}. However, there can be more than one weak solution which can also be non-Markov (see section 12.3 of \cite{stroock1979multidimensional}). In this work, under additional conditions on coefficients, we give a way of selecting weak solutions to~\eqref{eq:sde} that are Feller. This problem was initially studied in \cite{borkar2010new} where the authors give a procedure to select a Feller solution under appropriate assumptions, one of which is difficult to verify in practice. This work is an extension of \cite{borkar2010new}. In this work, we only use assumptions that can be directly verified as they only involve $b(\cdot)$, $\sigma(\cdot)$. In addition, we also establish that the selection procedure gives a unique Feller solution. 

A classical approach to selecting a Markov solution of the martingale problem in face of non-uniqueness is due to Krylov \cite{Krylov1973} (see also section 12.2 of \cite{stroock1979multidimensional}). This is based on successive minimization of a countable family of functionals on the solution measures, each one being minimized over the set of minimizers of the preceding one, and then arguing that the intersection of the nested family of minimizers thus obtained leads to a Markov solution. The problem with this approach is that no uniqueness is claimed, nor can the possibility of dependence on the specific choice of functionals to be minimized and the order thereof be ruled out.
	
	Usually, whenever there is an instance of non-uniqueness of solutions of any equation, the natural question that arises is: What is  the most relevant or physical / natural solution among the available ones? Kolmogorov (quoted in \cite[Pg. 626]{eckmann1985ergodic}) had suggested a way to answer this question \emph{viz.,}  perturb the equation with small noise and if this perturbed equation has a unique solution, then taking the noise to zero will give us the physical solution, if the corresponding limit exists and is unique (see \cite[Chapter 2]{flandoli2011random} for a detailed discussion). This point of view is adopted in various works in the literature. Here we give a very incomplete collection of examples: In \cite{kifer1974small}, author identified special invariant measures of the limiting smooth dynamical systems. For $n=1$ without diffusion, this problem is thoroughly studied under various set of assumptions in \cite{bafico1982small,veretennikov1983approximation,krykun2013peano,gradinaru2001singular,delarue2014transition,flandoli2014solution}. In \cite{gradinaru2001singular}, a large deviation principle is established for the one-dimensional case without diffusion and $b$ of the form $\text{sgn(x)}|x|^\gamma$, for $1>\gamma>0$ (a generalisation of this result can be found in \cite{herrmann2001phenomene}). For multidimensional case without diffusion, see \cite{delarue2019zero,zhang2012random}. In \cite{buckdahn2009limiting}, the authors studied the case with bounded measurable drift and showed that the solution to the limiting case lies the set of associated Fillipov solutions (see \cite{filippov2013differential}, for a definition).  The selection procedure in \cite{borkar2010new} is inspired by this philosophy and  is adopted in this work as well.

	\begin{remark}In what follows, by a solution of a  stochastic differential equation we always mean a solution to its corresponding martingale problem.\end{remark}

	\section{Selection Procedure}
	In this section, we provide a procedure that selects a unique Feller solution of~\eqref{eq:sde} under mild conditions. More precisely, under appropriate conditions on $b(\cdot)$ and $\sigma(\cdot)$, we consider a family of processes $\{X^\veps\}$ such that they are tight and converge to a limit. This limit will turn out to be the unique Feller solution of~\eqref{eq:sde}. We define $\calC_b(\RR^n)$ and $\calC^u_b(\RR^n)$ as the space of bounded continuous functions and space of bounded uniformly continuous functions on $\RR^n$ endowed with the topology of uniform convergence on compact sets, respectively.  
	
Since we are trying to study the Markovian nature of solutions of~\eqref{eq:sde}, it will be fruitful to consider the corresponding backward Kolmogorov equation \emph{i.e.,}  
\begin{align}\label{bke} \partial_t u -\calL(D^2u,Du,x)&=0 \textrm{ on $\calO\doteq  \RR^n\times (0,T)$,}\\
\label{bkeic}
u(x,0)&=f(x)\in \calC_b^u(\RR^n),
\end{align}
where, 
$$ \calL (M,p,x)\doteq  \frac{1}{2}\text{Tr}\left\{ \sigma(x)(\sigma(x))^\dagger M\right\} +  b(x).p, \text{ for $(M, p,x)\in \calS^n\times \RR^n\times\RR^n$}$$
with $\calS^n$  the space of symmetric  $n\times n$ matrices, $\text{Tr}(A)$  the trace of $n\times n$ matrix and $p.q$  the usual Euclidean inner product. Denote $\calO\cup \left( \RR^n\times \{0\}\right)$ as $\calO^*$.
\begin{remark}
	1.\ Note that the~\eqref{bke} and~\eqref{bkeic} is  written as an equivalent initial value problem rather than final value problem \emph{i.e.,} $u(x,t)=\EE\left[f(X_T)|X_{T-t}=x\right]$ (see Theorem \ref{th:FtoVS}  below). It is clear to see that these are equivalent ways of studying the above equation.

2.\	The $f$ above is chosen to be bounded uniformly continuous functions instead of just a bounded continuous function. As we shall in Proposition \ref{propdensity}, this choice does not cost us any generality. \end{remark}
Following \cite[Pg.\ 626]{eckmann1985ergodic} and \cite{borkar2010new}, for $\veps>0$, consider the following stochastic differential equation:
\begin{align}\label{eq:sdep}
dX^\veps_t=b^\veps(X_t^\veps)dt + \sigma^\veps(X^\veps_t)dW_t,\quad X^\veps_0\sim \mu ,
\end{align}
	where, $b^\veps$ and $\sigma^\veps $ satisfy the following assumption:
	\begin{assumption}\label{assu:pert}
		$b^\veps$ and $\sigma^\veps $ are bounded continuous,
		$$ b^\veps\to b \text { and } \sigma^\veps \to \sigma, \text{ uniformly on compact sets of $\RR^n$} $$
		and $\sigma^\veps(\sigma^\veps)^\dagger$ is positive definite, for every $\veps>0$. 
			\end{assumption}
		\begin{remark}
			In \cite{borkar2010new}, the authors considered $b^\veps\equiv b$ and $\sigma^\veps$ to be of the form
			$$ \sigma^\veps(\cdot)(\sigma^\veps(\cdot))^\dagger=\sigma(\cdot)\sigma(\cdot)^\dagger  +\veps I.$$
		Thus the selection procedure considered in \cite{borkar2010new} is a special case of the above setup. 
\end{remark}

It is well known that there exists a unique solution of~\eqref{eq:sdep} \cite[Theorem 7.2.1]{stroock1979multidimensional} and additionally,  that this solution is a Feller process. Consider the corresponding backward Kolmogorov equation:
\begin{align}\label{bkep}
\partial_t u^\veps - \calL^\veps (D^2u^\veps.Du^\veps,x)&=0, \textrm{ on $\calO$,}\\\label{bkepic}
u^\veps(x,0)&=f(x)\in \calC^u_b(\RR^n)
\end{align}
where, 
$$\calL^\veps(M,p,x)\doteq \frac{1}{2}\text{Tr}\left\{ \sigma^\veps (x)(\sigma^\veps(x))^\dagger M\right\} +  b^\veps(x).p , \text{ for $(M, p,x)\in \calS^n\times \RR^n\times\RR^n$}.$$
\begin{remark}
As mentioned above,~\eqref{bkep} and~\eqref{bkepic} are written as the initial value problem rather than final value problem. This is for the sake of the  ease of invoking already existing results about continuous viscosity solutions, which are usually for parabolic p.d.e.s in this form.
\end{remark}
 We are now in a position to describe the selection procedure: For $\{X^\veps\}_{\veps>0}$ defined as above
\begin{enumerate}
	\item [\namedlabel{limitstep1}{(1)}] Establish that $\{X^\veps\}_{\veps>0}$ is tight in $\calC$.
	\item[\namedlabel{limitstep2}{(2)}] Show that one of the limit points in law,  $X^*$  (say), is a Feller process. 
	\item [\namedlabel{limitstep3}{(3)}] Show that the limit point $X^*$ is in fact a solution of~\eqref{eq:sde}.
	\item [\namedlabel{limitstep4}{(4)}] Finally, show that there is a unique Feller solution of~\eqref{eq:sde}.
\end{enumerate} 
Through the  above, we would have established that the given selection procedure picks out the unique Feller process from the set of solutions of~\eqref{eq:sde}. We again remind the reader that the above selection procedure was already described in \cite{borkar2010new}, but under the assumptions that are not all easily verifiable. Also, in step~\ref{limitstep2}, the authors show the Feller property for a sequential limit $X^*$, which raises the question concerning what other sequential limit points might be like. The goal of this paper is to perform steps~\ref{limitstep1},~\ref{limitstep21},~\ref{limitstep3} and~\ref{limitstep4}, where step~\ref{limitstep21} is
\begin{enumerate}
	 \item[\namedlabel{limitstep21}{($\bar2$)}] Show that one of the limit point $X^*$ is Feller process and that  any limit point shares the same finite dimensional distribution as that of this Feller process.
\end{enumerate}
 With the intention of making this paper self-contained, we present the results of \cite{borkar2010new} that are used in this work. 
 \begin{definition}
 	The set of all points $x\in \RR^n$ such that there exists a neighbourhood $U$ of $x$ such that 
 	$$ \|b(y)-b(z)\|+\|\sigma(y)-\sigma(z)\|\leq K_U(x)\|y-z\|, \text { for every $y,z\in U$.}$$
 	
 It	is denoted by $L_{b,\sigma}$ and its complement by $NL_{b,\sigma}\doteq (L_{b,\sigma})^c$.
 \end{definition}
\begin{assumption}\label{assu:zero}
 	$b(\cdot):\RR^n\rightarrow \RR^n$ and $\sigma(\cdot):\RR^n\rightarrow \RR^n\times \RR^m$ (allowed to be degenerate) are bounded and \editcol{satisfy the following:  There are constants $C_b$ and $C_\sigma>0$ such that 
 	\begin{enumerate}
 		\item\label{i1} For $x,y\in \RR^n$, $\|b(x)-b(y)\|\leq C_b\|x-y\|^\alpha$ and $\|\sigma(x)-\sigma(y)\|\leq C_\sigma\|x-y\|^\beta.$
 		\item\label{i2} For every $x\in NL_{b,\sigma}$, there is $r>0$ (depending on $x$) such that $b(x)=0$, $\sigma(x)=0$ and $$C_\sigma\|x-y\|^{2\beta}\|v\|^2\geq v^\dagger\sigma(y)\left(\sigma(y)\right)^\dagger v\geq C_\sigma^{-1}\|x-y\|^{2\beta}\|v\|^2,$$
 		for $y\in B_r(x)\cap L_{b,\sigma}$. Here, $B_r(x)$ is the open ball of radius $r$ around $x$.
 		\item \label{i3}$\alpha$ and $\beta$ are such that \editcol{$1+\alpha-2\beta>0$} and $\beta>\frac{1}{2}$.
 	\end{enumerate}}

 	 \end{assumption}
  
\begin{remark}
Assumption~\ref{assu:zero} covers a large class of H\"{o}lder continuous drifts and diffusion coefficients, going beyond what was achieved in \cite{borkar2010new}. This assumption will allow us to construct a class of viscosity supersolutions (see Definition~\eqref{defn}) with desired properties (Lemma~\ref{specialsol}) in the following way: 
\begin{itemize}
	\item \eqref{i1} of Assumption~\ref{assu:zero} is mainly used because derivatives of function $\|x-x_*\|^\gamma$, for some $0<\gamma<1$ and $x_*\in \RR^n$ are again in the similar form \emph{viz.,} $\|x-x_*\|^{\gamma-1}$ at $x\neq x_*$. This fact is used in comparing the terms in~\eqref{eq:compterms}.
	\item \eqref{i2} of Assumption~\ref{assu:zero} is used to arrive at the estimate~\eqref{eq:estcomp}.
	\item In \eqref{i3} of Assumption~\ref{assu:zero}, $\beta>\frac{1}{2}$ is assumed to ensure that the convergence in~\eqref{uc2} and~\eqref{uc3} holds. Assuming $1+\alpha-2\beta>0$ ensures that the third term in~\eqref{eq:compterms} dominates the first two terms and thereby making the sum of the first three terms of~\eqref{eq:compterms} is non-negative in a neighbourhood of any $x_*\in NL_{b,\sigma}$.
\end{itemize}	
\end{remark} 
We give a few examples of pairs $(b,\sigma)$ satisfying Assumption~\ref{assu:zero} where~\eqref{eq:sde} can be shown to have non-unique weak solutions: The examples that we give are for $n=1$, but one can easily construct examples for $n>1$. However, showing non-uniqueness of weak solutions may be difficult.
\begin{enumerate}
	\item $b(x)={|x|^{\alpha}}$,  $\sigma(x)=0$ and $0<\alpha<1$. It can be easily seen there are multiple solutions for $\mu=\delta_{0}$, where, $\delta_x$ denotes the Dirac delta measure at $x\in \RR^n$. 
	\item \cite[Exercise 5.2.17]{karatzas2012brownian} $b(x)=3x^{\frac{1}{3}}$, $\sigma(x)= 3x^{\frac{2}{3}}$. For $\mu=\delta_0$, it is known that~\eqref{eq:sde} has multiple strong solutions in this case. This pair of $(b,\sigma)$ satisfy both~\ref{i1} and~\ref{i2} of Assumption~\ref{assu:zero}. However, $\alpha+1-2\beta=0$ in this case. This will not be an issue as Lemma~\ref{specialsol} can still be shown to hold by considering the explicit forms of $b$, $\sigma$ and optimum values (which are easy to find in this case) of all the relevant constants in the proof.
\end{enumerate}

We will show later that the above assumption implies that there is a unique Feller solution of~\eqref{eq:sde} \emph{i.e.,} Step ~\ref{limitstep4} of the procedure described above.

 \begin{theorem}
 	Under Assumption~\ref{assu:pert}, $\{X^\veps\}_{\veps>0}$ is tight in $\calC$.
 \end{theorem}
\begin{proof}
	 From the definition of $X^\veps$, for $0\leq s\leq t\leq T$, we have
	 \begin{align*}
	 X^\veps_t-X_s^\veps= \int_s^t b^\veps(X^\veps_u)du +\int_s^t \sigma^\veps(X^\veps_u) dW_u 
	 \end{align*}
	 Now consider
	 \begin{align*}
	 \|X^\veps_t-X^\veps_s\|^4&\leq 8\|\int_s^tb^\veps(X^\veps_u)du\|^4 + 8 \|\int_s^t \sigma^\veps(X^\veps_u) dW_u\|^4\\
&\leq 8K^4|t-s|^4 +24 K^4 |t-s|^2, \text{ from \cite[Lemma 4.1]{liptser1977statistics}}, 
	 	 \end{align*}
where, $K\doteq \max\{\|b\|_\infty,\|\sigma\|_\infty\}$.

	Therefore, using \cite[Theorem 12.3]{billingsley1968convergence}, we conclude the desired tightness.	 
\end{proof}
To proceed further, we introduce the notion of continuous viscosity solutions (see \cite{crandall1992user} for an excellent survey).  For a locally bounded function $w: \calO\rightarrow \RR$ and $(x,t)\in \calO$, we define $\calP_{\calO}^{2,+} w(x,t)$ as
\begin{align*}
\calP_{\calO}^{2,+} w(x,t)\doteq \Big\{ (a,p,X)\in& \RR\times \RR^n\times \calS^n: \text{ for } \calO\ni(y,s)\to (x,t), \text{ we have }\\
\quad  w(y,s)&\leq w(x,t) + a (s-t) +  p.(y-x)\\
& \quad \quad + \frac{1}{2}  (y-x)^\dagger X(y-x) + o\left(|t-s| +\|x-y\|^2\right)\Big\}
\end{align*}
and  $\calP_{\calO}^{2,-} w\doteq - \calP_{\calO}^{2,+} (-w)$. Also
\begin{align*}
\overline{\calP}_\calO^{2,+} w(x,t)\doteq \Big\{& (a,p,X)\in \RR\times \RR^n\times \calS^n: \exists (a_n,p_n,X_n, (x_n,t_n))\in \RR\times \RR^n\times \calS^n\times \calO\\
&\quad \text{ such that } (a_n,p_n,X_n)\in \calP_\calO^{2,+}(x_n,t_n) \text{ and } \\
&\quad(a_n,p_n,X_n, (x_n,t_n),u(x_n,t_n)) \xrightarrow{n\to\infty} (a,p,X, (x,t),u(x,t)) \Big\} 
\end{align*}
and $\overline{\calP}_\calO^{2,-} w(x,t)$ is defined analogously.
\begin{definition}\label{defn}
	A upper semicontinuous function $u:\calO\rightarrow \RR$ is called a viscosity subsolution of~\eqref{bke} and~\eqref{bkeic}  if 
	\begin{align}\label{subsol}
	a-\calL(X,p,x)\leq 0 , \text{ for $(x,t)\in \calO$ and  }  (a,p,X)\in \calP_\calO ^{2,+}u(x,t).
	\end{align} 
	and $u(x,0)\leq f(x)$.

	A lower semicontinuous function $u:\calO\rightarrow \RR$ is called a viscosity supersolution of~\eqref{bke} and~\eqref{bkeic} if 
	\begin{align}\label{supersol} a-\calL(X,p,x)\geq 0 \text{, for $(x,t)\in \calO$ and } (a,p,X)\in \calP_\calO ^{2,-}u(x,t).
	\end{align}
	and $u(x,0)\geq f(x)$.
	
	A function $u:\calO\rightarrow \RR$ is called a continuous viscosity solution of~\eqref{bke} and~\eqref{bkeic} if it is both a viscosity subsolution and a viscosity supersolution.
\end{definition}
\begin{remark}
1.\ 	Similar definitions hold for~\eqref{bkep} and~\eqref{bkepic}.

2.\ Equivalent definitions of viscosity subsolution and supersolution can be found in \cite[Pg. 1237-1238]{lions1983optimal}.

3.\	In what follows, sometimes we refer to viscosity subsolutions (viscosity supersolutions, respectively) as just subsolutions (supersolutions, respectively). Also, we refer to an upper semi-continuous function (a lower semi-continuous function, respectively) on $\calO$ as a subsolution (supersolution, respectively) even if it satisfies only~\eqref{subsol} (~\eqref{supersol}, respectively).
\end{remark}
Before we proceed, we recall the following result from \cite{borkar2010new}:
\begin{theorem}\label{th:FtoVS}If there is a Feller solution to~\eqref{eq:sde}, then there is a continuous viscosity solution to ~\eqref{bke} and~\eqref{bkeic}. The continuous viscosity solution is given by 
	$$ u(x,t)=E[f(X_T)|X_{T-t}=x].$$
\end{theorem}
\begin{remark}
	Analogous statement also holds for~\eqref{eq:sdep},~\eqref{bkep} and~\eqref{bkepic}.
\end{remark}
Using the theorem above, we can establish the existence of $u^\veps$ using a probabilistic argument. Thus we have:
\begin{corollary}\label{existencepert}
	Under Assumption~\ref{assu:pert}, there exists a continuous viscosity solution to~\eqref{bkep} and~\eqref{bkepic}.
\end{corollary}
\begin{proof}
	To see this, we note that~\eqref{eq:sdep} has a unique weak solution (denoted by $X^\veps$) and hence it is a Feller solution. Therefore from the above remark, there is a continuous viscosity solution given by 
	$$ u^\veps(x,t)=E[f(X^\veps_T)|X^\veps_{T-t}=x].$$ 
\end{proof}
We will later see that under the appropriate conditions, a Feller solution corresponds to a continuous viscosity solution and vice versa.

As mentioned already, we are interested in studying the limiting behavior of $X^\veps$ as $\veps\to 0$. Since $X^\veps$ is a Feller process, this can be done by studying the limiting behavior of $u^\veps=E[f(X^\veps_T)|X^\veps_{T-t}=x]$ as $\veps\to 0$. In general, we do not have strong uniform estimates (in $\veps$)  on $u^\veps$ that guarantee convergence. The best a priori estimate that we have in general is the following: $$ \| u^\veps(
\cdot,t)\|_\infty\leq \|f\|_\infty ,\text { for $0\leq t< T$}.$$
\begin{remark}
From the boundedness of  $b$ and $\sigma$, one can easily get the following continuity estimate uniform in $\veps$:
$$ \|u^\veps(x,t)-u^\veps(x,s)\|\leq K\|t-s\|, \text{ for $x\in \RR^n$ and $s,t\in[0,T)$}.$$
\end{remark}
Define functions $u^*$ and $u_*$ on $\calO^*$ as 
$$ {u}^*(x,t)\doteq \lim_{\delta\downarrow 0}\sup_{\veps>0}\left\{  u^\veps(y,s): \|x-y\|+|t-s|<\delta ,\; \veps<\delta\right\}$$
and 
$$ {u}_*(x,t)\doteq \lim_{\delta\downarrow 0}\inf_{\veps>0}\left\{  u^\veps(y,s): \|x-y\|+|t-s|<\delta ,\; \veps<\delta\right\}.$$
Clearly, $u^* \geq u_*$ and
\begin{align} \label{lowersemienv} \| u^*(
\cdot,t)\|_\infty\leq \|f\|_\infty \text{ and } \| u_*(
\cdot,t)\|_\infty\leq \|f\|_\infty,\text { for $0\leq t< T$}.
\end{align}
	From the previous remark and the definition of $u^*$ and $u_*$, it is easy to conclude that for any fixed $x\in \RR^n$, $u^*(x,\cdot)$ and $u_*(x,\cdot)$ are continuous functions.
 
\begin{proposition}\label{prop:subsol}
	Under Assumption~\ref{assu:pert} , $u^*$ and $u_*$ are subsolution and supersolutions of~\eqref{bke} and~\eqref{bkeic}, respectively.
\end{proposition} 
\begin{remark}
	The proof below is almost the same as the proof of \cite[Theorem V.1.7]{bardi1997optimal} and is given here for the sake of completeness.  We only show that $u^*$ is a subsolution of~\eqref{bke} and~\eqref{bkeic}, as showing that $u_*$ is a supersolution can be done in the similar way. 
\end{remark}
\begin{proof}
From an equivalent definition of viscosity subsolution (see \cite[Remark I.9]{lions1983optimal}), it suffices to show that if $(\bar x,\bar t)$ is a strict local maximum of $u^*-\phi$ for a twice continuously differentiable $\phi$ on $[0,T]\times \RR^n$, then 
	$$ \partial_t \phi(\bar x,\bar t) -\calL(D^2\phi(\bar x,\bar t), D\phi(\bar x,\bar t),\bar x )\leq 0 .$$
	As in  \cite[Lemma V.1.6]{bardi1997optimal}, we can show that there is a subsequence (still denoted by $\veps$) such that $(x^{\veps}, t^{\veps})$  is a local maximum of $  u^{\veps} -\phi$  and $(x^{\veps}, t^{\veps})\rightarrow (\bar x, \bar t)$. Using the fact that $ u^\veps$ is a subsolution, we get 
	\begin{align}
	\partial _t \phi (x^\veps, t^\veps) -\calL^\veps(D^2\phi(x^\veps, t^\veps),D\phi(x^\veps, t^\veps),x^\veps  )\leq0
	\end{align}
	Now taking $\veps\to 0$ and using uniform convergence on compact sets of $\calL^\veps$ to $\calL$, we get the result. 
	Finally, from \cite[Proposition VII.5.1]{fleming2006controlled} we have
	$$ u^*(x,0)=f(x), \text{ for } x\in \RR^n.$$
	This proves that $u^*$ is a viscosity subsolution of~\eqref{bke} and~\eqref{bkeic}.
\end{proof}
The functions $u^*$ and $u_*$ are the possible candidates for continuous viscosity solution of~\eqref{bke} and~\eqref{bkeic}. If $u^*=u_*\doteq \bar u$ on $\calO$, then clearly, $\bar u$ is the desired continuous viscosity solution. Even though $u^*\geq u_*$ trivially, we cannot say that $u^*\leq u_*$ (and hence $u^*=u_*$) in general without a comparison principle which says that under appropriate conditions, for a subsolution $u$ and supersolution $v$ of~\eqref{bke}, we have $u(x,t)\leq v(x,t)$ on $\calO$ whenever $u(x,0)\leq v(x,0)$ on $\RR^n$. Before we proceed to establish the aforementioned comparison principle, we give a very important consequence of equality of $u^*$ and $u_*$ on $\calO$. The following lemma follows as in
\cite[Lemma 1.9]{bardi1997optimal}.
\begin{lemma}\label{convunifcomp}
Suppose $u^*=u_*=\bar u$. Then	$$ u^\veps \to \bar u, \text{ uniformly on the compact sets of $\calO^*$}.$$
\end{lemma}

The comparison principle given below is in greater generality than what is required for our purposes.
\begin{lemma}(Comparison principle)\label{comparison}
	Let $u$ and $v$ be bounded sub and supersolutions of~\eqref{bke}, respectively  and that either  $u(\cdot,0)\in \calC^u_b(\RR^n)$ or  $ v(x,0)\in \calC^u_b(\RR^n)$. Suppose Assumption~\ref{assu:zero} holds.	Then 
	\begin{align} \label{eq:comp}
	\sup_{\calO^*}\left[u-v\right]=\sup_{x\in \RR^n}\left\{\left[u(x,0)-v(x,0)\right]\vee 0\right\}\end{align}
\end{lemma}
\begin{remark}
1.\ The proof adapts the techniques of \cite[Theorem II.9.1]{fleming2006controlled} and \cite[Theorem 1.3.2]{zhan2000viscosity}.	

2.\	Note that $u^*$ and $u_*$ clearly satisfy the hypothesis of this lemma. Therefore, this lemma implies that $u^*=u_*$.
\end{remark}

\begin{proof}  
	
	The proof we give is by contradiction.	Before we proceed, we make the following transformation: $u(x,t)\mapsto e^{-\gamma t}u(x,t)$ and $v(x,t)\mapsto e^{-\gamma t}v(x,t)$ with $\gamma>0$. This helps
us later in establishing the contradiction. With abuse of notation, we denote the transformed $u$ and $v$ also by $u$ and $v$. By assuming that $u$ and $v$ are smooth, we can guess the equation that $u$ and $v$ satisfy to be 
\begin{align*}
\partial_t u+\gamma u- \calL(D^2u,Du,x)=0.
\end{align*}
The fact that this indeed is the case is given by \cite[Pg.\ 98, Lemma II.9.1]{fleming2006controlled} (\cite[Lemma II.9.1]{fleming2006controlled} is applicable only for the first order case, proof for second order case is exactly along the same lines). Since $\gamma>0$ is arbitrary, it suffices to show~\eqref{eq:comp} for the transformed $u$ and $v$.

	We split the proof into two cases: Recall that $\calO^*=\RR^n\times [0,T)$.
	\begin{enumerate}
		\item (Case 1) 
		$$ \sup_{\calO^* }\left[u-v\right] > u(x,t)-v(x,t),\; \forall (x,t)\in NL_{b,\sigma}\times [0,T).$$ 
		\item (Case 2)
		$$ \sup_{\calO^*}\left[u-v\right] = u(x_*,t_*)-v(x_*,t_*),\text{ for some $(x_*,t_*)\in NL_{b,\sigma}\times [0,T)$}.$$ 
	\end{enumerate}	
	
	\noindent {\bf CASE 1: } We assume the contrary to~\eqref{eq:comp} \emph{i.e.,} 
	\begin{align}\label{contC1}\sup_{\calO^*}\left[u-v\right]-\delta>\sup_{x\in \RR^n}\left\{\left[u(x,0)-v(x,0)\right]\vee 0\right\}, \end{align}
	for some $\delta>0$.
	
	We define $u^\rho$ as $u-\frac{\rho}{T-t}$.
	Now consider the following auxillary function: For $\alpha,\rho,\beta>0$, define $M_{\alpha,\rho,\beta}:\RR^n\times \RR^n\times [0,T)\rightarrow \RR$ by
	\begin{align}\label{def:M}
	M_{\alpha,\rho,\beta}(x,y,t)\doteq  u^\rho(x,t)-v(y,t)-\alpha \|x-y\|^2 -\beta \|x\|^2,
	\end{align}
	
	From the behavior of $-\beta\|x\|^2$, boundedness of $u,v$ and upper semicontinuity of $M_{\alpha, \beta, \rho}$, we know that there exists $(\hat x,\hat y,\hat t)\in{\RR^n}\times{\RR^n}\times [0,T) $ which is a maximizer of $M_{\alpha, \rho,\beta}$ on ${\RR^n}\times{\RR^n}\times [0,T)$. 	
	
	Using Lemma~\ref{auxifun} in the Appendix, we can ensure that, for small enough $\rho$ and $\beta$  
	$$ \sup_{{\RR^n}\times{\RR^n}\times [0,T)}M_{\alpha, \rho,\beta}>\eta>0,$$
	for some $\eta>0$.
	From the arguments of \cite[Pg. 916, Eq. 3.11]{ishii1994uniqueness}, we have
	\begin{align}\label{eq:alphalim}
	\lim_{\alpha\to\infty}	\varlimsup_{\beta,\rho\to 0} \alpha\|\hat x-\hat y\|^2 = 0
	\end{align}	
	
	From Equation~\eqref{a1}, we have \editcol{$(x_0,y_0,t_0)$ such that 
	\begin{align} M_{\alpha,\beta,\rho}(x_0,y_0,t_0)&\geq  \sup_{ \calO^*}\left[u(x,t)-v(x,t)\right]-5\epsilon, \text{ for small enough $\epsilon$}\\
	\implies \sup_{ \RR^n\times\RR^n\times [0,T]}\left[M_{\alpha,\beta,\rho}\right]&\geq \sup_{ \calO^*}\left[u(x,t)-v(x,t)\right]-5\epsilon.\end{align}}
	
	From Equation~\eqref{eq:alphalim}, 
	we have 
	$$ \lim_{\alpha\to\infty}	\varliminf_{\beta,\rho\to 0}\sup_{ \RR^n\times\RR^n\times [0,T)}\left[M_{\alpha,\beta,\rho}\right]=\sup_{ \calO^*}\left[u-v\right]> u(x,t)-v(x,t),\; \forall (x,t)\in NL_{b,\sigma}\times[0,T).$$
	\editcol{ The inequality above follows from the arbitrariness of $\epsilon$.}

	Therefore, we have shown that for  large enough $\alpha$ and small enough $\beta$ and $\rho$, there is a $r>0$ such that $\hat x, \hat y\notin B_r(NL_{b,\sigma})$. 

	Since from the assumption of the contrary of~\eqref{eq:comp} and choice of $\alpha, \beta, \rho$, $(\hat x,\hat y, \hat t)$ are interior points of $\calO^*$, using \cite[Theorem 8.3]{crandall1992user}, we know that there exist $a,b, X,Y$ such that 
	\begin{align*}
	(a,\alpha(\hat x-\hat y) , X)&\in \overline{\calP}^{2,+} (u^\rho(\hat x, \hat t) -\beta\|\hat x\|^2)\\
	(b,\alpha(\hat x-\hat y) , -Y)&\in \overline{\calP}^{2,-} v(\hat y, \hat t)
	\end{align*}
	such that $a-b=0$ and 
	\begin{align}\label{matineq}-3\alpha \begin{pmatrix}
	\I& 0\\
	0&\I
	\end{pmatrix}\leq \begin{pmatrix}
	X&0\\
	0&Y
	\end{pmatrix} \leq 3\alpha \begin{pmatrix}
	\I&-\I\\
	-\I & \I
	\end{pmatrix},\end{align}
where $\I$ is the identity matrix in $\RR^{n\times n}$. Using the fact that $u$ and $u$ are sub and supersolution respectively, we get
	$$ a+ \gamma\left(u^\rho(\hat x, \hat t) -\beta\|\hat x\|^2\right) -\calL(\hat x, \alpha(\hat x-\hat y)+2\beta \I, X+2\beta \I) \leq -\frac{\rho}{T^2}$$
	$$b+\gamma v(\hat y, \hat t)-\calL(\hat x, \alpha(\hat x-\hat y), Y)\geq 0.$$
	From above, we have
	\begin{align*}
	& \gamma\left(u(\hat x,\hat y)-v(\hat y,\hat t)\right) -\frac{\gamma \rho}{T-\hat t} -\gamma \beta \|\hat x\|^2-\beta\langle b(\hat x), \hat x\rangle -\beta Tr\left\{\sigma(\hat x)\sigma(\hat x)^\dagger \right\}\\
	&- \langle b(\hat x)-b(\hat y), \alpha (\hat x-\hat y)\rangle - Tr\left\{ \sigma(\hat x)\sigma(\hat x)^\dagger X-\sigma(\hat y)\sigma(\hat y)^\dagger Y\right\}\leq - \frac{\rho}{T^2},
	\end{align*}
	From~\eqref{eq:alphalim} and the fact that there is a $r>0$ such that  $\hat x,\hat y\notin B_r(NL_{b,\sigma})$ (for large enough $\alpha$ and small enough $\beta$ and $\rho$), we have a $K_r>0$ such that 
	\begin{align*}
\|b(\hat x)-b(\hat y)\|+	\|\sigma(\hat x)-\sigma (\hat y)\|\leq K_r\|\hat x-\hat y\|.
	\end{align*}
This gives us
	\begin{align*}
	&\gamma\left(u(\hat x,\hat y)-v(\hat y,\hat t)\right) -\frac{\gamma \rho}{T-\hat t} -\gamma \beta \|\hat x\|^2-\beta\langle b(\hat x), \hat x\rangle\\
	& -\beta Tr\left\{\sigma(\hat x)\sigma(\hat x)^\dagger \right\} +\frac{\rho}{T^2} \leq  3\alpha  \|\sigma(\hat x)-\sigma(\hat y)\|^2+ K_r \|\hat x-\hat y\|^2\\
	&\leq 3\alpha K_r \|\hat x-\hat y\|^2
	\end{align*}
	In the above, we used the fact that left multiplying~\eqref{matineq} with $$\Sigma\doteq \begin{pmatrix}
	\sigma(\hat x)\sigma(\hat x)^\dagger& \sigma(\hat y)\sigma(\hat x)^\dagger\\
	\sigma(\hat x)\sigma(\hat y)^\dagger&\sigma(\hat y)\sigma(\hat y)^\dagger
	\end{pmatrix}$$ and taking the trace gives us
	\begin{align*}
	Tr\left\{\sigma(\hat x)\sigma(\hat x)^\dagger X-\sigma(\hat y)\sigma(\hat y)^\dagger Y\right\}&\leq 3\alpha Tr\left\{ (\sigma(\hat x)-\sigma(\hat y))(\sigma(\hat x)-\sigma(\hat y))^\dagger \right\}\\
	&\leq 3\alpha \|\sigma(\hat x)-\sigma(\hat y)\|^2
	\end{align*}
	Finally, taking $\beta,\rho\downarrow 0$ and then $\alpha  \uparrow \infty$, we have a contradiction, since it leads to $$0<\gamma \left( \sup_{\calO^*}[u-v]\right)\leq 0.$$
	
	\noindent {\bf CASE 2: }
	
	We now have 
	\begin{align}\label{supNL}
		\sup_{\calO^*}\left[u-v\right] =  u(x_*,t_*)-v(x_*,t_*),
		\end{align}
		 for some $(x_*,t_*)\in NL_{b,\sigma}\times [0,T)$.
	 
	 We deal with this case in the following way: we select a very special subsolution  of~\eqref{bke} denoted by $\psi$ in order to understand the behavior of $u-v$ around $(x_*,t_*)$. This special subsolution $\psi$ is such that 
	 \begin{enumerate}
	 	\item \editcol{$u-\psi$ and $v+\psi$} are also subsolutions and supersolutions, respectively.
	 	\item $\psi(x,t)=h(\|x-x_*\|)$, for some modulus of continuity $h$.
	 \end{enumerate}  
 A few remarks are in order before we proceed. It is important to note that it is not apriori clear that $u+v$ is a subsolution whenever $u$ and $v$ are subsolutions. Hence, asking for \editcol{$u+\psi$ ($v-\psi$}, respectively) to be a subsolution (supersolution, respectively) is a non-trivial requirement. It will be clear from the proof that if we desired to relax the Assumption~\ref{assu:zero} to a much weaker assumption which says that $b$ and $\sigma $ are merely bounded uniformly continuous, then the current result will continue to hold as long as one can find a very special subsolution that satisfies the aforementioned properties. As constructing the subsolutions or supersolutions is much easier than constructing constructing viscosity solutions, this approach will hopefully helpful in studying cases where $b$ and $\sigma$ are merely bounded uniformly continuous. 
 
 We now continue with the proof. In the following lemma, we establish the existence of $\psi$. 
 \begin{lemma}\label{specialsol}For $K>0$ and $x_*\in NL_{b,\sigma}$,
 	$$ \psi=K \|x-x_*\|^\gamma$$ 
 	is a \editcol{supersolution} of~\eqref{bke} over a neighbourhood around $x_*$ whenever $\gamma>\max\{2(1-\beta),1-\alpha\}$. Additionally, \editcol{$u-\psi$ and $v+\psi$ are subsolution and supersolutions} of~\eqref{bke}, respectively.	
 \end{lemma}
 \begin{proof} From the linearity of~\eqref{bke}, it is clear that without loss in generality we can take $K=1$.
 	We first prove that for $\theta>0$
 	$$ \psi_\theta(x,t)\doteq \{\|x-x_*\|^2+\theta \}^{\frac{\gamma}{2}} $$ is a \editcol{supersolution} of~\eqref{xieq} below. \editcol{Before we do that, note that as $\theta\downarrow 0$, $\psi_\theta \to \psi$ uniformly on the compact sets of $\RR^n$.} Since $\psi_\theta$ is twice differentiable in $x$ and smooth in $t$, we can conclude that $\psi_\theta$ is a \editcol{supersolution} of~\eqref{xieq} by showing that it is a classical supersolution \emph{i.e.,} for $A(x) \doteq\frac{1}{2} \sigma(x)\sigma(x)^T$ with $x$ in a closed ball,
 	\begin{align}\label{xieq}
 	\calL_\theta(x) \doteq \partial_t \psi_\theta- b(x).\nabla \psi_\theta - Tr\{ A(x)\Delta \psi_\theta\} +3\xi = 0, 
 	\end{align}
 	for small  $\xi>0$ and sufficiently small $\theta$ (depending on $\xi$).
 	Firstly, note that from Assumption~\ref{assu:zero}, $b(x)=0$ and $A(x)=0$, whenever $x\in NL_{b,\sigma}$.  Therefore, $\calL_\theta (x) = 0$  whenever $x\in NL_{b,\sigma}$. This means that we only have to show that $\calL_\theta(x)  \geq 0$ for $x\in L_{b,\sigma}$.
 	\begin{align*}
 	& \partial_t \psi_\theta- b(x).\nabla \psi_\theta - Tr\{ A(x)\Delta \psi_\theta\}\\
 	&=- \frac{\gamma }{\{\|x-x_*\|^2+\theta \}^{1-\frac{\gamma}{2}}}b(x).(x-x_*)\\
 	&\quad -\frac{\gamma Tr\{A(x)\}}{\{\|x-x_*\|^2+\theta \}^{1-\frac{\gamma}{2}}} + \frac{2\gamma(1-\frac{\gamma}{2}) (x-x_*)^\dagger A(x)(x-x_*)}{\{\|x-x_*\|^2+\theta \}^{2-\frac{\gamma}{2}}}\\
 	&\editcol{\geq} -\frac{\gamma \editcol{C_b} }{\{\|x-y\|^2+\theta \}^{1-\frac{\gamma}{2}}}\|x-x_*\|^\alpha\|x-x_*\|\\
 	&\quad-\frac{\gamma Tr\{A(x)\}}{\{\|x-x_*\|^2+\theta \}^{1-\frac{\gamma}{2}}} + \frac{2\gamma(1-\frac{\gamma}{2}) (x-x_*)^\dagger A(x)(x-x_*)}{\{\|x-x_*\|^2+\theta \}^{2-\frac{\gamma}{2}}}\\
 	&\text{ (from Assumption~\ref{assu:zero})}\\
 	&\editcol{\geq}- \frac{\gamma \editcol{C_b} }{\{\|x-y\|^2+\theta \}^{1-\frac{\gamma}{2}}}\|x-x_*\|^\alpha\|x-x_*\|\\
 	&\quad-\frac{\gamma Tr\{A(x)\}}{\{\|x-x_*\|^2+\theta \}^{1-\frac{\gamma}{2}}} + \frac{2\gamma(1-\frac{\gamma}{2})Tr\{A(x)\} \|x-x_*\|^2}{n\{\|x-x_*\|^2+\theta \}^{2-\frac{\gamma}{2}}},\\
 	&\text{ from the fact that $\frac{1}{n}Tr(A(x))\leq \sup_{\|v\|=1}\{v^\dagger A(x)v\}$},
 	\end{align*}
 	In the above, we used 
 	$$ \frac{\partial}{\partial x_j} \frac{\gamma (x-x_*)_i}{\{\|x-x_*\|^2+\theta \}^{1-\frac{\gamma}{2}}}= \frac{\gamma \delta_{ij}}{\{\|x-x_*\|^2+\theta \}^{1-\frac{\gamma}{2}}} - \frac{2\gamma(1-\frac{\gamma}{2}) (x-x_*)_i(x-x_*)_j}{\{\|x-x_*\|^2+\theta \}^{2-\frac{\gamma}{2}}}.$$
 	From Assumption~\ref{assu:zero}, it is clear that  $$ \editcol{C_\sigma ^{-1}\|x-x_*\|^{2\beta}\leq} Tr\{A(x)\}\leq T(x)\doteq \editcol{\frac{C_\sigma}{n}}\|x-x_*\|^{2\beta},$$
 	for some $\editcol{C_\sigma}>0$.
 	Thus we have
 	\begin{align}\nonumber
 	\calL_\theta(x) \ &\editcol{\geq} \ \frac{-\gamma\editcol{C_b} }{\{\|x-x_*\|^2+\theta \}^{1-\frac{\gamma}{2}}}\|x-x_*\|^{1+\alpha}\\\nonumber
 	&\quad\quad\quad-\frac{\gamma \editcol{C_\sigma}\|x-x_*\|^{2\beta} }{n\{\|x-x_*\|^2+\theta \}^{1-\frac{\gamma}{2}}}+ \frac{2\gamma(1-\frac{\gamma}{2})\|x-x_*\|^{2\beta} \|x-x_*\|^2}{nC_\sigma\{\|x-x_*\|^2+\theta \}^{2-\frac{\gamma}{2}}}\\\nonumber
 	&\editcol{\geq\frac{1}{\{\|x-x_*\|^2+\theta \}^{1-\frac{\gamma}{2}}}\Bigg[-\gamma\editcol{C_b}\|x-x_*\|^{\alpha+1}-\frac{\gamma C_\sigma}{n}\|x-x_*\|^{2\beta}}.\\\label{eq:estcomp}
 	&\editcol{\quad\quad\quad+ \frac{2\gamma(1-\frac{\gamma}{2})}{nC_\sigma\{\|x-x_*\|^2+\theta \}}\|x-x_*\|^{2\beta+2}\Bigg]}
 	\end{align}
 	\editcol{ It is clear that if $\gamma> \max\{2(1-\beta),1-\alpha\}$, then $\alpha+\gamma-1> 0$, $2\beta+\gamma-1> 0$ and $2\beta+\gamma-2> 0$. This means that as $\theta\to 0$,
 		\begin{align}\label{uc1} \frac{\|x-x_*\|^{\alpha+1}}{\{\|x-x_*\|^2+\theta \}^{1-\frac{\gamma}{2}}}&\to \|x-x_*\|^{\alpha+\gamma-1},\\\label{uc2}
 		 \frac{\|x-x_*\|^{2\beta}}{\{\|x-x_*\|^2+\theta \}^{1-\frac{\gamma}{2}}}&\to \|x-x_*\|^{2\beta+\gamma-1},\\\nonumber
 		 &\text{and}\\\label{uc3}
 		\frac{\|x-x_*\|^{2\beta+2}}{\{\|x-x_*\|^2+\theta \}^{2-\frac{\gamma}{2}}}&\to \|x-x_*\|^{2\beta+\gamma-2}\end{align}
 		uniformly on compact sets of $\RR^n$. Therefore for small $\xi>0$, small enough $\theta$ and on a large closed ball of $\RR^n$,
 		\begin{align}\label{eq:compterms}
 		\calL_\theta(x)\geq -\gamma\editcol{C_b}\|x-x_*\|^{\alpha+\gamma-1}-\frac{\gamma C_\sigma}{n}\|x-x_*\|^{2\beta+\gamma-1}+ \frac{2\gamma(1-\frac{\gamma}{2})}{nC_\sigma}\|x-x_*\|^{2\beta+\gamma-2} -3\xi.
 		\end{align}
 		
 		Since $1+\alpha-2\beta> 0$ from Assumption~\ref{assu:zero}, we can conclude that on a neighbourhood around $x_*$, the sum of the first three terms in the above equations is non-negative and this implies that $\psi_\theta$ is a supersolution of~\eqref{xieq}.
 		Since $v$ is supersolution of~\eqref{bke}, it is automatically a supersolution of~\eqref{xieq}. This together with the smoothness of $\psi_\theta$ concludes that $v+\psi_\theta$ is supersolution of~\eqref{xieq} (see Lemma~\ref{sum}).
 		
 		Noting that $v+\psi_\theta\to v+\psi$, uniformly on compact sets, we can conclude that $v+\psi$ is a supersolution of~\eqref{bke} (from \cite[Lemma II.6.2]{fleming2006controlled}). We can conclude that $u-\psi$ is a subsolution of~\eqref{bke} by considering $\xi<0$ in~\eqref{xieq} and arguing similarly as above. This proves the lemma.}
 \end{proof}
 
  Now that we have proved that a $\psi$ with the desired properties exists, let us assume the contrary to~\eqref{eq:comp} \emph{i.e.,}
 \begin{align}\label{negsupNL}  u(x_*,t_*)-v(x_*,t_*)-\delta> \sup_{x\in \RR^n}\left\{\left[u(x,0)-v(x,0)\right]\vee 0\right\}, 
 \end{align}
 for some $\delta>0$.  
 In the following lemma, we take $h:\RR^+\rightarrow \RR^+$ as $h(r)\doteq K r^\gamma$, where $\gamma$ is as in Lemma~\ref{specialsol} and $K>0$.
\begin{lemma}
	For some $r>0$,
	$$ u(x,t)-v(y,t) -h(\|x-x_*\|)- h(\|y-x_*\|)\leq 0,$$
	whenever $\|x-x_*\|<r$, $\|y-x_*\|<r$ and $t\in[0,T)$.
\end{lemma}	
\begin{proof}
	From Lemma~\ref{specialsol}, we know that there exists $r>0$ such that $h(\|x-x_*\|)$ is a supersolution of~\eqref{bke} and $r$ can be chosen uniformly with respect to $K$. From now on, we fix such an	 $r>0$.  Choose $$K\doteq \frac{\sup_{0\leq t<T}\left\{\|u(\cdot,t)\|_\infty+\|v(\cdot,t)\|_\infty\right\}}{r^\gamma}.$$ 
	
Define 
	$$\Phi(x,y,t)\doteq u(x,t)-v(y,t)\editcol{-h(\|x-x_*\|)- h(\|y-x_*\|)} -\beta \left( \|x\|^2+\|y\|^2\right) - \frac{\rho}{T-t} $$ 
	and $$ \Delta\doteq \left\{ (x,y,t)\in \RR^n\times \RR^n\times[0,T):t>0, \|x-x_*\|< r \text{ and }\|y-x_*\|< r \right\}.$$
	
	From~\eqref{supNL}, \eqref{negsupNL} it is clear that there is $(\hat x,\hat y,\hat t)\in \Delta$ such that 
	\begin{align}\label{contrary} u(\hat x,\hat t)-v(\hat y,\hat t)\editcol{-h(\|\hat x-x_*\|)- h(\|\hat y-x_*\|)}>\eta>0.\end{align}
	
	Note that from the definition of $\Phi$, we know that there is a maximum $(\bar x_*,\bar y_*,\bar t_*)$ (depending on $\beta$ and $\rho$) of $\Phi$ on $\Delta$. For small enough $\beta$ and $\rho$ (Lemma~\ref{auxifun}), we can ensure
	$$ \Phi(\bar x_*, \bar y_*, \bar t_*)>0.$$
	From boundedness of $u$ and $v$ and sub-linearity of $h$, it is clear that 
	$$\beta(\|\bar x_*\|^2 +\|\bar y_*\|^2)<\infty \text{ and } \lim_{\beta\to 0}\beta(\|\bar x_*\|+\|\bar y_*\|)\to 0. $$
		It is clear that $0<\bar t_*<T$, for small enough $\beta$ and $\rho$.  To see that $\|\bar x_*-x_*\|<r$ and $\|\bar y_* -x_*\|<r$, for small enough $\beta$ and $\rho$, observe that from the choice of $K$, we have
	$$ h(\|x-x_*\|)+h(\|y-x_*\|)\geq \sup_{0\leq t<T}\left\{\|u(\cdot,t)\|_\infty+\|v(\cdot,t)\|_\infty\right\},$$
	whenever $\|x-x_*\|=r$ or $\|y-x_*\|=r$.
\editcol{	From this, it follows that 
	$$ \Phi(\bar x_*,\bar y_*,\bar t_*)\leq -\beta (\|\bar x_*\|^2+\|\bar y_*\|^2)-\frac{\rho}{T-\bar t_*}$$
	 This is contradictory to the assumption as $(\bar x_*,\bar y_*,\bar t_*)$ is the maximum of $\Phi$ and 
	 $$ \Phi(\bar x_*, \bar y_*, \bar t_*)>0.$$} 
	Therefore, for small enough $\beta$ and $\rho$, $(\bar x_*,\bar y_*,\bar t_*)\in \Delta$.
	
	Now using \cite[Theorem 8.3]{crandall1992user}, we know that there exist $a,b, X,Y$ such that 
	\begin{align*}
	(a,\beta\bar x_* , X)&\in \overline{\calP}^{2,+} (u^\rho(\bar x_*, \bar t_*) \editcol{-h(\|\bar x_*- x_*\|)})\\
	(b,\beta \bar y_* , -Y)&\in \overline{\calP}^{2,-}( v(\bar  y_*, \bar t_*) \editcol{+ h(\|\bar y_*-x_*\|)})
	\end{align*}
	such that $a=b$ and for every $\epsilon=\beta^{-1}$,
	$$-(\frac{1}{\epsilon}+ \beta ) \begin{pmatrix}
	\I& 0\\
	0&\I
	\end{pmatrix}\leq \begin{pmatrix}
	X&0\\
	0&Y
	\end{pmatrix} \leq (\beta+\epsilon \beta^2) \begin{pmatrix}
	\I&0\\
	0 & \I
	\end{pmatrix},$$
	where, $\I$ is the identity matrix in $\RR^{n\times n}$. Using the fact that $u-h(\|\cdot-x_*\|)$ and $v+h(\|\cdot-x_*\|)$ are sub and supersolution, respectively, we get
	$$ a+ \gamma\left(u^\rho(\bar x_*, \bar t_*) \editcol{-h(\|\bar x_*- x_*\|)}\right) -\calL(\bar  x_*, 2\beta \bar x_*, X) \leq -\frac{\rho}{T^2}$$
	$$b+\gamma \left(v(\bar  y_*, \bar t_*)\editcol{+ h(\|\bar y_*-x_*\|)}\right)-\calL(\bar y_*, 2\beta \bar y_*, Y)\geq 0.$$
	From above, we have
	\begin{align*}
	& \gamma\left(u(\bar x_*,\bar t_*)-v(\bar y_*,\bar t_*)\editcol{-h(\|\bar x_*-x_*\|)- h(\|\bar y_*-x_*\|)}\right) -\beta\langle b(\bar x_*), \bar x_*\rangle +\beta\langle b(\bar y_*), \bar y_*\rangle\\
	& - Tr\left\{ \sigma(\bar x_*)\sigma(\bar x_*)^\dagger X-\sigma(\bar y_*)\sigma(\bar y_*)^\dagger Y\right\}\leq - \frac{\rho}{T^2},
	\end{align*}
	Taking $\beta$ and $\rho$ to $0$, we have a contradiction, from the boundedness of $b$ and $\sigma$.
\end{proof}

From the lemma, we know that
$$ u(x,t)-v(y,t) \editcol{-h(\|x-x_*\|)- h(\|y-x_*\|)}\leq 0,$$
whenever $\|x-x_*\|<r$, $\|y-x_*\|<r$ and $t\in [0,T)$. In particular, 
$$ u(x_*,t_*)-v(y_*,t_*)\leq 0.$$ This contradicts~\eqref{negsupNL}. Hence, we have

$$\sup_{\calO^*}\left[u-v\right]=\sup_{x\in \RR^n}\left\{\left[u(x,0)-v(x,0)\right]\vee 0\right\}. $$  

This concludes the proof of Lemma \ref{comparison}.
	 
\end{proof}

\editcol{ To summarize, until now we have proved that $u^*$ ($=u_*$)  is a continuous viscosity solution whenever $f\in \calC^u_b(\RR^n)$. In the following proposition, we extend this result to allow for $f$ to lie in $\calC_b(\RR^n)$. 
\begin{proposition}\label{propdensity}
 Suppose $u^f$ stands for $u^*$ when $u^*(\cdot,0)=f\in \calC_b(\RR^n)$. Then $u^f$ is the continuous viscosity solution of~\eqref{bke}.
\end{proposition}
\begin{proof}
	First, we show that for $f_1,f_2\in \calC^u_b(\RR^n)$, 
	\begin{align}\label{linearityzero}\text{ (linearity) } u^{f_1+f_2}=u^{f_1}+u^{f_2}.\end{align}
	From Lemma~\ref{comparison} and~\ref{convunifcomp}, we know that for $0\leq t<T$ as $\veps \to 0$,
	\begin{align}\label{lim3} \|u^{\veps,f_1}(\cdot,t)-u^{f_1}(\cdot,t)\|_\infty,\; \|u^{\veps,f_2}(\cdot,t)-u^{f_2}(\cdot,t)\|_\infty, \|u^{\veps,f_1+f_2}(\cdot,t)-u^{f_1+f_2}(\cdot,t)\|_\infty\to 0.\end{align}
	Here, $u^{\veps,g}$ denotes the continuous viscosity solution of~\eqref{bkep} whenever $u^{\veps,g}(\cdot,0)=g$. It is clear from the proof of Corollary~\ref{existencepert} that for any $\veps>0$, $u^{\veps,g}$ is linear in $g$, for $g\in \calC_b(\RR^n)$. Together with~\eqref{lim3}, the aforementioned linearity of $u^{\veps,g}$ gives us~\eqref{linearityzero}. 
	
To keep the notation standard, let us write $u^f(x,t)$ as $(P_t f)(x)$. From~\eqref{lowersemienv}, we have 
$$ \|P_tf\|_\infty\leq \|f\|_\infty, \text{ for $0\leq t<T$}.$$
To prove the statement of the proposition, we note that $\calC^u_b(\RR^n)$ is dense in $\calC_b(\RR^n)$ under the topology of uniform convergence on compact sets of $\RR^n$. Hence, $P_t$ can be extended continuously to $\calC_b(\RR^n)$ \emph{viz.,} for $f\in \calC_b(\RR^n)$ and $0\leq t<T$
$$ P_tf\doteq \lim_{n\to \infty} P_tf_n,$$
whenever $\{f_n\}_{n\geq 1}\subset \calC^u_b(\RR^n)$ such that $f_n \to f$ as $n\to \infty$ uniformly on compact sets of $\RR^n$. In fact, the above limit is uniform over any compact set of $[0,T)$.	  
\end{proof}}

From Theorem~\ref{th:FtoVS}, we know that every Feller process gives a corresponding continuous viscosity solution. The following result from \cite{borkar2010new} gives the converse of this statement \emph{viz.,} a continuous viscosity solution of~\eqref{bke} gives a corresponding Feller process.

\begin{theorem}{\cite[Theorem 3.2]{borkar2010new}}\label{th:limit}
Suppose Assumption~\ref{assu:pert} and the statement of Lemma~\ref{comparison} hold. Then all the limit points (in the sense of weak convergence) are solutions to the martingale problem corresponding to~\eqref{eq:sde} and one of the limit points is Feller. Also, if $\bar X$ and $\tilde X$ are two limit points, then for $0\leq s\leq t\leq T$, $(\bar X_s,\bar X_t)$ and $(\tilde X_s, \tilde X_t)$ has the same laws.	
\end{theorem}

The above theorem does not prevent the limit points from having different finite dimensional distributions. The following result says that all limit points share the same finite dimensional distribution.
\begin{theorem}
Under Assumptions~\ref{assu:pert} and~\ref{assu:zero}, for any $k\in \NN$, $\{f_i: \|f_i\|_\infty\leq 1\}_{i=1}^k \subset \calC_b(\RR^n)$ and $0\leq t_1<t_2<\ldots<t_k<T$, we have 
	\begin{align*}
	\lim_{\veps \to 0} \EE_{X^\veps} \left[ \prod_{i=1}^k f_i(X^\veps_{t_i})\right]= \EE_{X^*} \left[ \prod_{i=1}^k f_i(X^*_{t_i})\right],
	\end{align*}
	where, $X^*$ is the Feller process obtained in Theorem~\ref{th:limit}. 
\end{theorem}
\begin{remark}
	In \cite{borkar2010new}, the authors showed that this result holds for $k\leq 2$ \emph{i.e.,} for any $f_1,f_2\in \calC_b(\RR^n)$, we have
	\begin{align}
	\lim_{\veps\to 0}	\EE_{X^\veps}\left[ f_1(X^\veps_s)f_2(X^\veps_t)\right]= \EE_{\bar{X}}\left[ f_1(\tilde X_s)f_2(\tilde X_t)\right], \; s\leq t
	\end{align}
	for every (weak) limit point $\tilde X$ of $\{X^\veps\}_{\veps>0}$.  
\end{remark}
\begin{proof}
	Let $u^{\veps,f}$ (and $u^{0,f}$) be the continuous viscosity solutions of~\eqref{bkep} with initial condition $f\in \calC_b(\RR^n)$ (and~\eqref{bke} with initial condition $f\in \calC_b(\RR^n)$).
	
	The only obstacle in proving the desired result for $k>2$  is to show that 
	$$ \lim_{\veps \to 0} \EE_{X^\veps} \left[ \prod_{i=1}^k f_i(X^\veps_{t_i})\right] \text{ exists} $$
	for any $\{f_i\}\in\calC_b(\RR^n)$ (set of bounded continuous functions on $\RR^n$) and $0\leq t_1<t_2<t_3<\ldots<t_k\leq T$. We will only show that the above limit exists and equals
	$$ \EE_{X^*}\left[ f_1(X^* _{t_1})f_2(X^* _{t_2})f_3(X^* _{t_3})\right].$$
	for $k=3$. Extension to $k>3$ case follows along the similar lines. 
	
	Let $P^\veps(t,A,y)$ (with $0\leq t<T$, $y\in \RR^n$ and $A\in \calB(\RR^n)$)  be the corresponding transition kernel for $X^\veps$. Using Lemma~\ref{comparison} and~\ref{convunifcomp} together, we know that for any $f\in \calC_b(\RR^n)$, 
	$$ u^{\veps,f}(x,t)=\int_{ \RR^n} f(y)P^\veps(t,dy, x)\to \bar u(x,t) \text{ uniformly on compact sets of $\calO^*$,}$$
where $\bar u$ is defined in Lemma~\ref{convunifcomp}.
	
For $0\leq t<T$, \cite[Theorem 3.1]{borkar2010new} gives us the existence of measure $P^0(t,dy,x)$ such that 
	$$ u^{0,f}(x,t)= \int_{\RR^n}f(y)P^0(t,dy,x).$$
	
	For $t_1>t_2>t_3$, consider 
 \begin{eqnarray*}
\lefteqn{\EE\left[ f_1(X^\veps _{t_1})f_2(X^\veps _{t_2})f_3(X^\veps _{t_3})\right] =}\\
&& \int f_1(x_1)P^\veps(t_2-t_1, dx_1,x_2)f_2(x_2)P^\veps(t_3-t_2,dx_2,x_3)f_3(x_3) P^\veps_{}(t_1,dx_1,x)\mu(dx).
\end{eqnarray*}
	Consider
$$	h^\veps_1(x_2,t_2) \ \doteq \	\int_{x_1\in \RR^n} f_1(x_1)P^\veps(t_2-t_1, dx_1,x_2)
	 \ = u^{\veps,f_1}(x_2,t_2). $$
	From Lemma~\ref{convunifcomp} , we know that 
	$$ h^\veps_1 \rightarrow h^0_1 \doteq \int_{x_1\in \RR^n} f_1(x_1)P^0(\cdot-t_1 , dx_1,\cdot ), \text{ uniformly on compact sets of $\calO^*$. }$$ 
	Now define
	\begin{align} h^\veps_2 (x_3,t_3)&\doteq \int_{x_2\in\RR^n} h^\veps_1 (x_2,t_2)f_2(x_2)P^\veps (t_3-t_2,dx_2,x_3)\\
	&= u^{\veps, (h^\veps_1 f_2)}(x_3,t_3)
	\end{align} 
	We note that $(h^\veps_1 f_2)$ converges uniformly on compact sets of $\RR^n$ to $h^0_1f_2$ and that for every $\delta>0$ and $0\leq t<T$, there is a compact set $\calK$ (depending on $t$) such that $P^\veps (t,\calK,x)>1-\delta$. Noting that $\|f_i(\cdot,t)\|_\infty \leq 1$ and $\|h_i(\cdot,t)\|_\infty\leq 1$, we have
	\begin{align*}
	&\left|u^{\veps, h^\veps_1 f_2 }(x_3,t_3)-u^{0, h^0_1f_2 }(x_3,t_3)\right|\\
	&\leq \left| \int_{x_2\in\calK} h^\veps_1 (x_2,t_2)f_2(x_2)P^\veps (t_3-t_2,dx_2,x_3)- \int_{x_2\in\calK} h^0_1 (x_2,t_2)f_2(x_2)P^0 (t_3-t_2,dx_2,x_3)\right|\\
	&\quad  + 2\delta\\
	&\leq \left|\int_{x_2\in\calK} h^\veps_1 (x_2,t_2)f_2(x_2)P^\veps (t_3-t_2,dx_2,x_3)-\int_{x_2\in\calK} h^0_1 (x_2,t_2)f_2(x_2)P^\veps (t_3-t_2,dx_2,x_3)\right| \\
	& +\left|\int_{x_2\in\calK} h^0_1 (x_2,t_2)f_2(x_2)P^\veps (t_3-t_2,dx_2,x_3) - \int_{x_2\in\calK} h^0_1 (x_2,t_2)f_2(x_2)P^0 (t_3-t_2,dx_2,x_3)\right|\\
	&\quad +2\delta.
	\end{align*} 
	The first  and second terms above go to zero as $\veps \to 0$ and as $\delta$ is arbitrary, we have 
	$$ h_2^\veps (x_3,t_3) \rightarrow h^0_2(x_3,t_3)\doteq \int_{x_2\in\RR^n} h^0_1 (x_2,t_2)f_2(x_2)P^0(t_3-t_2,dx_2,x_3).$$
	Since $t_1>t_2>t_3$ are fixed throughout and all the functions are evaluated only at one of these time instants, all convergence claims can be assumed to be uniform in them. So we only worry about uniform convergence in $x$, the first argument. From the above calculation, only pointwise convergence of $h_2^\veps(\cdot, t_3)$ can be inferred. We now show that 
the convergence of $h_2^\veps$ to $h^0_2$ is uniform on compact sets of $\RR^n$. Recall that  
	$$ u^{\veps,f}\rightarrow u^{0,f}, \text{ uniformly on compact sets of $\RR^n\times[0,T)$,}.$$
	for every $f\in \calC_b(\RR^n)$. 
	In other words, for every $\rho>0$ and compact $\calJ\subset \RR^n$, there is a $\veps_0>0$ such that for $\veps<\veps_0$, we have
	\begin{align*}
	|	\int_{\RR^n} f(y)P^\veps(t,dy,x)-\int_{\RR^n} f(y)P^0(t,dy,x)|<\rho, \text{ for $\forall x\in \calJ$.} 
	\end{align*}
	
	The following set is a relatively compact in $\calP(\RR^n)$:
	$$ \Pi\doteq \{ P^\veps (t, \cdot, x): \veps>0,\; x\in \calJ\}.$$
	Indeed, consider a sequence $(\veps_n,x_n)$ such that $\{\veps_n>0\}$ and $\{x_n\in \calJ\}$. Then there exists a subsequence (still denoted by $n$) $(\veps_{n},x_{n})$ such that $x_{n}\to x^* \in \calJ$ and there exists a sub-subsequence (still denoted by $n$) such that $P^{\veps_{n}} (t, \cdot, x^*)$ converges weakly to $P^{0} (t, \cdot, x^*)$. Along the subsequence $(\veps_{n},x_{n})$ and for $g\in \calC_b(\RR^n)$, consider
	\begin{align}
	&\left|	\int_{\RR^n} g(y)P^{\veps_n}(t,dy,x_n)-\int_{\RR^n} g(y)P^0(t,dy,x^*)\right|\\
	&\leq \left|	\int_{\RR^n} g(y)P^{\veps_n}(t,dy,x_n)-\int_{\RR^n} g(y)P^{0}(t,dy,x_n)\right|\\
	&\quad\quad+ \left|	\int_{\RR^n} g(y)P^{0}(t,dy,x_n)-\int_{\RR^n} g(y)P^0(t,dy,x^*)\right|.
	\end{align}
	As $n\to \infty$, second term on the right hand side goes to zero due to Feller continuity of $X^*$ and the first term goes to zero due to the uniform convergence on compact sets of $\RR^n$. Hence $\Pi$ is relatively compact. Thus for any $\delta>0$, there exists a compact set $\calK_{\delta,\calJ}  $ (only depending on $\delta$ and $\calJ$) such that 
	$$ P^\veps(t,\calK_{\delta,\calJ},x)>1-\delta, \text { for every $\veps>0$ and $x\in \calJ$} $$
	This finally shows that $h_2^\veps\rightarrow h_2^0$ uniformly on compact sets. Similarly, defining 
	$$ h_3^\veps( x)\doteq \int_{x\in \RR^n} h_2^\veps(x_3,t_3)f_3(x_3) P^\veps(t_3,dx_3,x)$$ and proceeding as above, we can conclude that 
	$$ h_3^\veps \rightarrow h_3^0 \doteq \int_{x\in \RR^n} h_2^0(x_3,t_3)f_3(x_3) P^0(t_3,dx_3,x) \text{ uniformly on compact sets of $\RR^n$}.$$ 
	We conclude that
	\begin{align}
	\lim_{\veps\to 0}\EE_{X^\veps}\left[ f_1(X^\veps _{t_1})f_2(X^\veps _{t_2})f_3(X^\veps _{t_3})\right] &= \EE_{X^*}\left[ f_1(X^* _{t_1})f_2(X^* _{t_2})f_3(X^* _{t_3})\right]\\
	&= \EE_{\bar{X}}\left[ f_1(\bar X _{t_1})f_2(\bar X _{t_2})f_3(\bar X _{t_3})\right],
	\end{align}
	where $\bar X$ is any other limit point of $X^\veps$.
	
	This proves the theorem.
\end{proof}
\begin{remark}
	From Lemma~\ref{comparison}, uniqueness of continuous viscosity solution of~\eqref{bke} and~\eqref{bkeic} follows. This means that among all the solutions of~\eqref{eq:sde}, there exists a unique Feller solution.
\end{remark}

Even though we have assumed appropriate H\"older continuity of $b$ and $\sigma$, we expect that just uniform continuity of $b$ and $\sigma$ might suffice for our results to hold. We however, could not provide a proof of this statement. Relaxing this condition further gives rise to cases where no Feller solution exists for the corresponding martingale problem. Here is one such counterexample.
\begin{counterexample}{\cite[Example 12.4.2]{stroock1979multidimensional}}
	For $n=1$, $\sigma\equiv 0$ and 
	\begin{align}
	b(x)=
	\begin{cases}&	\text{(sgn(x))}|x|^{\frac{1}{2}},\text{ for $|x|\leq 1$}\\
	&  1,\text{ for $x\geq 1$}	\\
	& - 1,\text{ for $x\leq -1$}
	\end{cases}
	\end{align}
	It can be easily seen that~\eqref{eq:sde} with the above coefficients does not have a Feller solution. In particular, the transition kernel fails to be continuous at $x=0$. For more such examples, see \cite{bafico1982small} and \cite[Pg. 746]{borkar2010new}. 
\end{counterexample}

 \appendix
 \section{}
 	\begin{lemma}\label{sum}
 	If $u$ and $v$ are  viscosity subsolutions of~\eqref{bke} and $v$ is $C^{1,2}((0,T)\times \RR^n)$, then $u+v$ is also a subsolution. 
 \end{lemma}
 
 \begin{proof} 	
 	From the definition of viscosity subsolution, we know that when $\calP_\calO^{2,+} u(x,t) \neq \phi$,
 	\begin{align} \label{eq:subsol}
 	a- \calL(X,p,x)\leq 0, \text{ for $(t,x)\in (0,T)\times \RR^n$ and $(a,p,X)\in \calP_\calO^{2,+} u(x,t)$},
 	\end{align}
  	\begin{align}\label{eq:supsol}
 	b- \calL(Y,q,x)\leq  0, \text{ for $(t,x)\in (0,T)\times \RR^n$ and $(b,q,Y)\in \calP_\calO^{2,+} v(x,t)$}
 	\end{align}
 	From the differentiability of $v$ and definition of $\calP_\calO^{2,+}$, it is clear that $\calP_\calO^{2,+} v(x,t)\neq  \emptyset$, for every $(x,t)\in (0,T)\times \RR^n$ and more importantly, is a singleton. Let $(c,r,Z)\in \calP_\calO^{2,+}(u+v)(x,t)$. It is clear that $(c-b, r-q, Z-Y)\in \calP_\calO^{2,+} u(x,t)$ and 
 	\begin{align*}
 	c-b-\calL(Z-Y,r-q,x)\leq 0, \\
 	c-\calL(Z,r,x)\leq b-\calL(Y,q,x)\leq 0, 
 	\end{align*}
 	from~\eqref{subsol} and the linearity of $\calL$. Finally, from the definition of viscosity subsolution, we have the result.	
 \end{proof} 
 In the rest of the appendix, we assume the conditions of the statement of Lemma~\ref{comparison}. Let $M_{\alpha, \beta, \rho}$ be as defined in~\eqref{def:M} and $(\hat x,\hat y, \hat t)$ is the maximizer of $M_{\alpha, \beta, \rho}$ on ${\RR^n}\times{\RR^n}\times [0,T]$.

 \begin{lemma}\label{auxifun}\cite[Pg. 29]{zhan2000viscosity}
 Suppose~\eqref{contC1} holds. Then there exists  $\beta_0$ and $\rho_0$ such that for every $\forall\alpha>0$, $\beta<\beta_0$ and $\rho<\rho_0$, we  have the following: 	
	$$ \sup_{{\RR^n}\times{\RR^n}\times [0,T]}M_{\alpha, \rho,\beta}>\eta>0.$$
 \end{lemma}
\begin{proof}
	Let 
	$$\gamma \doteq \lim_{r\to0}\sup\left\{ u(x,t)-v(y,t): \|x-y\|<r, (x,y,t) \in  \RR^n\times \RR^n \times [0,T]\ \right\}.$$
	From~\eqref{contC1}, clearly we have,  
	$$ \gamma \geq \sup_{ \RR^n\times [0,T]}\left[u(x,t)-v(x,t)\right]> \sup_{\RR^n}\left\{\left[u(x,0)-v(x,0)\right]\vee 0\right\}\doteq M_b.$$
	
	For $\epsilon>0$, there is $r_0$ such that for $r<r_0$, we 
	have
	$$\sup\left\{ u(x,t)-v(y,t): \|x-y\|<r, ( x,y,t) \in \RR^n\times \RR^n \times [0,T]\ \right\}>\gamma-\epsilon $$
		To ensure that $\alpha\|x_0-y_0\|^2$ is small, choose $r< \min\{\sqrt{\frac{\eps}{2\alpha}},r_0\}$. Now there exists $(x_0,y_0,t_0)\in \RR^n\times \RR^n \times [0,T]$ such that 
	$$ u(x_0,t_0)-v(y_0,t_0)+\epsilon >\gamma-\epsilon \text{ and } \alpha\|x_0-y_0\|^2<\eps.$$
	It is also clear that there exists $\beta_0$ and $\rho_0$ such that $\frac{\rho}{T-t_0}<\eps$ and $\beta \|x_0\|^2<\eps$, for every $\beta<\beta_0$ and $\rho<\rho_0$. To summarize, we have shown that 
	\begin{align} 
	M_{\alpha, \beta, \rho}(x_0,y_0,t_0)&= u(x_0,t_0)-v(y_0,t_0)-\alpha\|x_0-y_0\|^2-\frac{\rho}{T-t_0}-\beta \|x_0\|^2\\\label{a1}
	&> \gamma-5\eps\\
	&>M_b, \text{ for small enough $\eps$.}
	\end{align}
	This concludes the proof.

\end{proof}

	Note that $\hat x$, $\hat y$ and $\hat t$ depend on $\alpha$, $\beta$ and $\rho$. We then have:

\begin{lemma}\label{lem:a}
	$$ \lim_{\alpha\to\infty} \varlimsup_{\beta,\rho\to 0} \alpha \|\hat x -\hat y\|^2 =0.$$

\end{lemma}
\begin{proof}
	From Lemma~\ref{auxifun}, $\sup_{\RR^n\times \RR^n \times [0,T]}M_{\alpha,\beta,\rho}>0$ and this along with the boundedness of $u$, $v$ means that 
	$$ \alpha\|\hat x-\hat y\|^2 +\beta\|\hat x\|^2 + \frac{\rho}{T-\hat t}\leq \|u\|+\|v\|\doteq M.$$
	This immediately gives us
	\begin{align*}
	\|\hat x-\hat y\|\leq \sqrt{\frac{M}{\alpha}}, \; \beta\|\hat x\|^2\leq M \text{ and } \frac{\rho}{T-\hat t}\leq M, 
	\end{align*}
	then
	\begin{align*}
\sup_{\RR^n\times \RR^n\times[0,T]}M_{\alpha,\beta,\rho}-\sup_{\RR^n\times \RR^n\times[0,T]}M_{\frac{\alpha}{2},\frac{\beta}{2},\frac{\rho}{2}}&\leq 
	M_{\alpha,\beta,\rho}(\hat x,\hat y, \hat t)- M_{\frac{\alpha}{2},\frac{\beta}{2},\frac{\rho}{2}}(\hat x,\hat y,\hat t)\\
	&= -\frac{\alpha}{2}\|\hat x-\hat y\|^2 -\frac{\beta}{2}\|\hat x\|^2 -\frac{\rho}{2(T-\hat t)}
	\end{align*} 
	From Lemma~\ref{auxifun}, we know that $M_{\alpha,\beta,\rho}$ is bounded away from zero from below  uniformly for small enough $\beta$ and $\rho$ and all $\alpha$. From monotone convergence theorem, 
	$$ \lim_{\alpha\to \infty } \lim_{\beta,\rho\to 0}\left( \alpha \|\hat x -\hat y\|^2 + \beta\|\hat x\|^2 + \frac{\rho}{T-\hat t}\right)=0$$	
	and we are done.
\end{proof}
\begin{lemma}
	For $u, v$ as in Lemma~\ref{comparison}, there exists  $\alpha_0$, $\beta_0$ and $\rho_0$ such that for every $\forall\alpha>\alpha_0$, $\beta<\beta_0$ and $\rho<\rho_0$, we  have the following: 	
	$$ \sup_{{\RR^n}\times{\RR^n}\times [0,T]}M_{\alpha, \rho,\beta} > \sup_{\RR^n}\left\{[u(x,0)-v(x,0)]\vee 0\right\}.$$
\end{lemma}
\begin{proof}
	Suppose $v$ is uniformly continuous. Assume the contrary to the statement of the lemma. Then there exists a sequence $(\alpha_n,\beta_n,\rho_n) \to (\infty, 0 ,0)$ such that the corresponding $\hat t=0$ (as $\hat t\neq T$ because of the term involving $\rho$). It is easy to see from definition of $M_{\alpha,\beta,\rho}$, Lemma~\ref{auxifun} and Lemma~\ref{lem:a} that
	\begin{align*}
	M_b<M_{\alpha_n,\beta_n,\rho_n}(\hat x,\hat y,0)&\leq u(\hat x,0)-v(\hat y,0)\\
	&\leq  v(\hat x,0)-v(\hat y,0)+ \sup_{\RR^n}\left\{[u(\hat{x},0)-v(\hat{x},0)]\vee 0\right\}.
\end{align*}
Letting $n \to \infty$ and using uniform continuity of $v$, we have 	$M_b < M_b$,	a contradiction.
\end{proof}
\noindent \textbf{Acknowledgement:} {The work of ASR is funded by Institute Postdoctoral Fellowship, Indian Institute of Technology, Bombay. The work of VSB was supported in part by an S.\ S.\ Bhatnagar Fellowship from the Council for Scientific and Industrial Research, Government of India.}

	\bibliographystyle{aomplain}	
	\bibliography{Uniqueness_viscosity_solutions}

\end{document}